\documentclass[11pt]{amsart}

\usepackage[left=3.5cm, right=3.5cm, top=3cm, bottom=3.5cm]{geometry}

\usepackage{vaucanson-g}

\usepackage[latin1]{inputenc}
\usepackage{amsfonts}
\usepackage{amsmath}
\usepackage{amssymb}
\usepackage{amsthm}
\usepackage[english]{babel}

\usepackage[T1]{fontenc}
\usepackage{mathrsfs}

\usepackage{enumerate}
\usepackage{amssymb}
\usepackage{amsmath}
\usepackage{pstricks}

\usepackage{graphicx}
\usepackage[breaklinks, dvips, ps2pdf]{hyperref}

\newtheorem{lemma}{Lemma}[section]

\newtheorem{prop}{Proposition}[section]

\newtheorem{defi}{D\'efinition}[section]
\newtheorem{theorem}{Theorem}[section]
\newtheorem{definition}{Definition}[section]
\newtheorem{coro}{Corollary}[section]
\theoremstyle{definition}

\newtheorem{rem}{\textit{Remark}}[section]
\newtheorem{ex}{Example}[section]

\def\P{\mathcal{P}}
\def\A{\mathcal{A}}
\def\F{\mathbb{F}}
\def\N{\mathbb{N}}

\def\K{\mathbb{K}}

\def\Q{\mathbb{Q}}
\def\a{{\bf a}}
\def\b{{\bf b}}
\def\t{{\bf t}}

\begin{document}

\title[] {Rational approximations to algebraic Laurent series with coefficients in a finite field}


\author{Alina Firicel}

\address{Universit\'{e} de Lyon\\
Universit\'{e} Lyon 1 \\
 Institut Camille Jordan \\
  UMR 5208 du CNRS \\
   43, boulevard du 11 novembre 1918 \\ 
   F-69622 Villeurbanne Cedex, France}
   

\maketitle 


\begin{abstract} 
In this paper we give a general upper bound for the irrationality exponent of algebraic Laurent series with coefficients in a finite
field. Our proof  is based on a method introduced in a different
framework by Adamczewski and Cassaigne. It makes use of automata theory
and, in our context, of a classical theorem due to Christol. We then
introduce a new approach which allows us to strongly improve this
general bound in many cases. As an illustration, we give few examples of
algebraic Laurent series for which we are able to compute the exact
value of the irrationality exponent.
\end{abstract}

\begin{section}{Introduction}

One of the basic question in Diophantine approximation is how well real numbers can be approximated by rationals. The theory of rational approximation of real numbers has then been transposed to function fields after the pioneering works of Maillet \cite{Maillet} in 1906 and Gill \cite{Gill} in 1930.
 In the present work, we are interested in the way algebraic Laurent series with coefficients in a finite field  can be approximated by rational functions.

Given a field $\mathbb K$, we let $\mathbb K(T)$, $\mathbb K[[T^{-1}]]$ and $\mathbb K((T^{-1}))$ denote, respectively, 
the field of rational functions, the ring of formal series and the field of Laurent series over the field 
$\mathbb K$. We also consider the absolute value defined on $\mathbb K(T)$ by $$\vert P/Q \vert=\vert T\vert^{\deg P- \deg Q},$$ for $(P, Q) \in \mathbb K[T]^2$, where $\vert T\vert$ is a fixed real number larger than $1$. The field of Laurent series in $1/T$, usually denoted 
by $\mathbb K((T^{-1}))$, 
should be seen as a completion of the field $\mathbb K(T)$ for this absolute value. 
Thus, if $f$ is a nonzero element of $\mathbb K((T^{-1}))$ defined by
$$f(T)= a_{i_0}T^{-i_0}+a_{i_0+1}T^{-i_0-1}+\cdots, $$ where $ i_0\in \mathbb Z, a_i\in \mathbb K, 
a_{i_0}\neq 0, $ we have $|f|=\vert T\vert^{-i_0}$.
We also say that a Laurent series is algebraic if it is algebraic over the field of rational functions $\K(T)$. The degree of an algebraic Laurent series  $f$ in $\mathbb K((1/T))$ is defined as $[\mathbb K(T)(f):\mathbb K(T)]$, the degree of the field extension generated by $f$.

We recall that the irrationality exponent (or measure) of a given Laurent series $f$, denoted by $\mu(f)$, 
is the supremum of the real numbers $\tau$ for which the inequality
$$ \left| f-\frac{P}{Q}\right | < \frac{1}{\vert Q \vert ^\tau}$$
has infinitely many solutions $(P, Q) \in \K[T]^2$, $Q\neq 0$.
Thus, $\mu(f)$ measures the quality of the best rational approximations to $f$.

\medskip

Throughout this paper, $p$ denotes a prime number and $q$ is a power of $p$. Our aim is to study the irrationality exponent of  algebraic Laurent series in the field $\F_q((1/T))$.

 In 1949, Mahler \cite{mahler} observed that the analogue of the fundamental Liouville's inequality holds true
for algebraic Laurent series over a field of positive characteristic.

\begin{theorem}(Mahler, 1949) Let $\K$ be a field of positive characteristic and $f \in \K((T^{-1}))$ 
be an algebraic Laurent series over $\K(T)$ of degree $d>1$. Then, there exists a positive real number $C$ 
such that $$ \left| f-\frac{P}{Q}\right |\geq \frac{C}{\vert Q \vert ^d}, $$
for all $(P, Q)\in \K[T]^2$, with $Q\neq 0$.
 \end{theorem}

In other words, Mahler's theorem tells us that the irrationality exponent of an algebraic irrational 
Laurent series is at most equal to its degree. 

In the case of real numbers, Liouville's theorem was superseded by the works of Thue \cite{Thue09}, Siegel  \cite{Siegel}, 
Dyson \cite{Dyson} and others, leading to the famous Roth's theorem \cite{Roth55}, which tells that the irrationality exponent of an algebraic real number is equal to $2$. In 1960, Uchiyama obtained an analogue of Roth's theorem \cite{Uchiyama}, for the case of Laurent series with coefficients in a field of characteristic $0$. 

When the base field has positive characteristic, it is well-known that there is no direct analogue of Roth's theorem. In fact, the Liouville--Mahler theorem turns out to be optimal. In order to see this, it is sufficient to consider the element $ f_q \in \F_q((T^{-1}))$ defined by
 $f_q(T)=\sum_{i\geq 0}{T^{-q^i}}$. It is not difficult to see that $f_q$ is an algebraic Laurent series of degree 
 $q$ (since it verifies the equation $f^q-f+T^{-1}=0$) while the irrationality exponent of $f$ is equal to 
 $q$. Note that these examples are sometimes referred to as Mahler's algebraic Laurent series. 
In the same direction, Osgood \cite{osgood1} 
and Baum and Sweet \cite{baumsweet} gave examples of algebraic Laurent series of various degrees for which Liouville's bound is the best possible.
For a special class of algebraic Laurent series, the bound given by Liouville for the irrational exponent 
was improved by Osgood \cite{osgood1, osgood2}. In 1976, this author proved an analog of 
Thue's theorem for algebraic Laurent series which are not solutions of a rational Ricatti differential equation. 
In 1996, de Mathan and Lasjaunias \cite{lasjaunias_demathan} proved that Thue's theorem actually 
holds for every algebraic Laurent series in $\K((T^{-1}))$, $\K$ being an arbitrary field of characteristic $p$, 
which satisfies no equation of the form $f=(Af^{p^s}+B)/(Cf^{p^s}+D)$, where $A, B, C, D \in \K[T]$, not all zero, and $s\in \N^*$. Laurent series satisfying such an equation are called hyperquadratic and they were studied by many authors \cite{Lasjaunias, Schmidt2, Thakur2, Voloch88}. Note that every hyperquadratic Laurent series 
does also satisfy a Ricatti differential equation.

\bigskip

Except the results obtained by Mahler, Lasjaunias-DeMathan or Osgood,  we do not know other general method in order to bound up the irrationality exponent of algebraic Laurent series over $\F_q(T)$. It is worth mentionning that the situation for function fields totally differs from the one of real numbers. For instance Schmidt \cite{Schmidt2} and Thakur \cite{Thakur2} independently proved that the set of irrationality exponents of algebraic Laurent series contains all rational real numbers greater than or equal to $2$.  

The aim of this paper is to introduce a new approach in order to bound up the irrationality exponent of an algebraic Laurent series, which is based on the use of the Laurent series expansion. As a starting point, we use a  theorem of Christol \cite{christol_2} which characterize in terms of automata the algebraic Laurent series with coefficient in a finite field. 
More precisely, we recall that $f(T)=\sum_{i\geq 0}{a_iT^{-i}} \in \F_q[[T^{-1}]]$ is algebraic 
if and only if
the sequence $(a_i)_{i\geq 0}$ is generated by a $p$-automaton. Furthermore, we recall that 
by a classical result of Eilenberg \cite{eilenberg} the so-called $p$-kernel of a $p$-automatic sequence 
is always finite (see Section \ref{Maximal repetitions in automatic sequences}). 

Our main result is the following explicit general upper bound for the irrationality exponent of 
algebraic Laurent series in $\F_p((1/T))$.

\begin{theorem}\label{MUIN11}
Let $f(T)=\sum_{i\geq -k}{a_iT^{-i}}$ be an algebraic Laurent series with coefficients in a finite field of 
characteristic $p$. Let $s$ be the cardinality of the $p$-kernel of 
$\a=(a_i)_{i\geq 0}$ and $e$ be the number of states of the minimal automaton generating $\a$ (in \textit{direct reading}). 
Then the irrationality exponent $\mu(f)$ satisfies 
\begin{equation}
 \mu(f) \leq p^{s+1}e.
\end{equation}
\end{theorem}
%



The approach we use to prove Theorem \ref{MUIN11} 
already appears in a different framework in \cite{adamczewski_bugeaudR, adamczewski_cassaigne, adamczewski_rivoal}. 
It is essentially based on repetitive patterns occurring in automatic sequences. More precisely, 
each algebraic formal series $f(T)=\sum_{i\geq 0}{a_iT^{-i}}$ is identified with a $p$-automatic sequence 
$\a:=(a_i)_{i\geq 0}$ over $\F_q$. Then we use a theorem of Cobham which characterizes $p$-automatic 
sequences in terms of $p$-uniform morphisms of free monoids (see Section \ref{3.Morphisms and Cobham's theorem}). As a consequence of this result 
and of the pigeonhole principle, we are able to find 
infinitely many pairs of finite words $(U_n, V_n)$ and a number $\omega >1$ such that $U_nV_n^{\omega}$ is prefix of $\a$ for every positive integer $n$.  Hence, there exists an infinite sequence of pairs 
of polynomials $(P_n, Q_n)$ such that the Laurent series expansion of the rational function $P_n/Q_n$ 
is the (ultimately periodic) sequence $\textbf{c}_n:=U_nV_n^{\infty}$. 
The sequence of rational functions $P_n/Q_n$ provides good rational approximations to $f$ since 
the words $\a$ and $\textbf{c}_n$ have the common prefix $U_nV_n^{\omega}$. Furthermore, the length of $U_n$ 
and $V_n$ are respectively of the form $kp^n$ and $\ell p^n$. Using such approximations 
we are able to prove the following result (see Theorem \ref{muin112}): 

\begin{equation}
 \frac{k+\omega\ell}{k+\ell}\leq \mu(f) \leq \frac{p^{s+1}(k+\ell)}{(\omega-1)\ell}.
 \end{equation}

In practice, it may happen that we can choose $U_n$ and $V_n$ such that $\a$ and $\textbf{c}_n$ 
have the same first $(k+\omega \ell)p^n$ digits, while the $(k+\omega \ell)p^n+1$th are different. 
In this case, the upper bound we obtain for the irrationality exponent may be a much better one and 
in particular does not depend anymore on the cardinality of the $p$-kernel. 
Furthermore, in the case where we can prove that $(P_n, Q_n)=1$ for all $n$ large enough, we obtain 
a significative improvement on the upper bound, as it will be explained in Remark \ref{3.sameprefix}, that sometimes leads to the exact value $\mu(f)$. 
Note that, when working with similar constructions involving real numbers, it is well-known that 
this coprimality assumption is usually difficult to check (see \cite{adamczewski_rivoal}). \\

In the second part of this paper,
we introduce a new approach in order to overcome this difficulty. We provide an algorithm 
that allows us to check, in a finite amount of time, wether the polynomials $P_n$ 
and $Q_n$, associated with an algebraic Laurent series $f$, are relatively prime 
for all $n$ large enough. In order to do this, we observed that the rational approximations 
we obtained have a very specific form: the roots of $Q_n$ can only be $0$ or the $\ell$th roots 
of unity (see Section \ref{3.Equivalentcondition}). 
Then we have to develop a calculus allowing to compute the polynomials $P_n(T)$. 
 In order to do this, we introduce some matrices associated with $p$-morphisms. These matrices generalize 
 the so--called incidence matrix of the underlying morphism (see Section \ref{Matrix associated with morphisms}) 
 and their study could also be of independent interest.\\

In the last part of this paper, we illustrate the relevance of our approach with few examples. 
We give in particular several algebraic Laurent series for which we are able to compute the exact value 
of the irrationality exponent.


\end{section}

\begin{section}{Terminology and basic notions}

A \emph{word} is a finite, as well as infinite, sequence of symbols (or letters) belonging to a nonempty set $\mathcal A$, called \emph{alphabet}. We usually denote words by juxtaposition of their symbols. 

Given an alphabet $\mathcal A$, we let $\mathcal{ A}^*:= \cup_{m=0}^{\infty} \mathcal {A}^m$ denote the set of finite words over $\mathcal A$. Let $V:=a_0a_1\cdots a_{m-1} \in \mathcal{ A}^*$; then, the integer $m$ is the length of $V$ and is denoted by $\left|V\right|$. The word of length $0$ is the empty word, usually denoted by $\varepsilon$. We also let $\mathcal{A}^m$ denote the set of all finite words of length $m$ and by $\mathcal{A}^{\N}$ the set of all infinite words over $\mathcal A$. We typically use the uppercase italic letters $ U, V, W$ to represent elements of $\mathcal{ A}^*$. We also use bold lowercase letters $\a, \bf b, \bf c$ to represent infinite words.

%
%


For any nonnegative integer $n$, we denote by $U^{n}:=UU\cdots U$ ($n$ times concatenation of the word $U$). More generally, for any positive real number $\omega$ we let $U^{\omega}$ denote the word $U^{\lfloor \omega \rfloor}U'$, where $U'$ is the prefix of $U$ of length $\lceil(\omega-\lfloor \omega \rfloor)\vert U \vert \rceil$. Here,  $ \lfloor \zeta \rfloor$ and $\lceil \zeta \rceil$ denote, respectively, the integer part and the upper integer part of the real number $\zeta$. We also denote $U^{\infty}:=UU\cdots $, that is $U$ concatenated (with itself) infinitely many times.

An infinite word $\a$ is \emph{periodic} if there exists a finite word $V$ such that $\a=V^{\infty}$. An infinite word is \emph{ultimately periodic} if there exist two finite words $U$ and $V$ such that $\a=UV^{\infty}$. 

Throughout this paper, we set $\A_m:=\{0, 1, \ldots, m-1\}$, for $m\geq 1$, which will serve as a generic alphabet.

\begin{subsection}{Automatic sequences and Christol's theorem}\label{Automatic sequences and Christol's theorem}

Let $k$ be a positive integer, $k\geq 2$. An infinite sequence $\a=(a_i)_{i\geq 0}$ is $k$-automatic if, roughly speaking, there exists a finite automaton which produces the term $a_i$ as output, when the input is the $k$-ary expansion of $i$.

For a formal definition of an automatic sequence, let us define a \textit{$k$-deterministic finite automaton with output} or, shortly, $k$-DFAO. This is a $5$-tuple
$$M=(Q,  \delta, q_0, \Delta, \varphi)$$
where $Q$ is a finite set of states, $ \delta: Q\times \A_k \mapsto Q$ is the transition function, $q_0$ is the initial state, $\Delta$ is the output alphabet and $\varphi: Q \mapsto \Delta$ is the output function.
For a finite word $W=w_r w_{r-1} \cdots w_0 \in \A_k^r$, we let $[W]_k$ denote the number $\sum_{i=0}^{r}{w_i k ^i}$.

We now say that a sequence $\a=(a_i)_{i\geq 0}$ is \textit{$k$-automatic} if there exists a $k$-DFAO such that $a_i=\varphi(\delta(q_0, w))$ for all $i\geq 0$ and all words $W$ with $[W]_k=i$.

\medskip

A classical example of automatic sequence is the so-called Thue-Morse sequence: $\t=(t_i)_{i\geq 0}=01101001100 \cdots$, which counts the number of $1$'s (mod $2$) in the base-$2$ representation of $i$. It is generated by the following automaton. More references on automatic sequences can also be found in the monograph \cite{Allouche_Shallit}.

\begin{figure}[h]
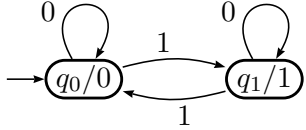

\centering
\VCDraw{%
  \begin{VCPicture}{(0, -0.5)(4, 1.5)}
\StateVar[q_0/0]{(0, 0)}{a} \StateVar[q_1/1]{(4, 0)}{b}
\Initial[w]{a}
\LoopN{a}{0}
\LoopN{b}{0}
\ArcL{a}{b}{1}
\ArcL{b}{a}{1}
\end{VCPicture}}
\caption{Automaton generating the Thue-Morse sequence}
\label{AB:figure:thue}
\end{figure}


If we now consider the Laurent series $f_{\t}(T)=\sum_{i\geq 0}{t_i T^{-i}}$ as an element of $\F_2((T^{-1}))$, one can check that $f_{\t}$ satisfies the algebraic equation:
$$(T+1)^3 f_{\t}^2(T)+T	(T+1) f_{\t}(T)+1=0.$$
Hence, $f_{\t}$ is an algebraic Laurent series over $\F_2(T)$ whose sequence of coefficients is a $2$-automatic sequence. Actually, this is not an isolated case. Indeed, there is a famous theorem of Christol \cite{christol_2} which precisely describes the algebraic Laurent series over $\F_q(T)$ as follows.

\begin{theorem} [Christol, 1979] Let $f_{\a}(T)=\sum_{i\geq -i_0}{a_iT^{-i}}$ be a Laurent series with coefficients in a finite field of characteristic $p$. Then $f_{\a}$ is algebraic over $\F_q(T)$ if and only if the sequence $\a=(a_i)_{i\geq 0}$ is $p$-automatic.
\end{theorem}

We also mention that a well-known result of Eilenberg proves that a sequence is $p$-automatic if and only if is $q$-automatic, for any $q$, power of $p$.

\end{subsection}

\subsection{Morphisms and Cobham's theorem}\label{3.Morphisms and Cobham's theorem}

Let $\mathcal A$ (respectively $\mathcal B$) be a finite alphabet and let $\mathcal{A}^*$ (respectively $\mathcal{B}^*$) be the corresponding free monoid. A morphism $\sigma$ is a map from $\mathcal{A}^*$ to $\mathcal{B}^*$ such that $\sigma(UV)=\sigma(U)\sigma(V)$ for all words $U, V \in \mathcal{A}^*$. Since the concatenation is preserved, it is then possible to define a morphism defined on $\mathcal A$.

Let $k$ be a positive integer. A morphism $\sigma$ is said to be \textit{$k$-uniform} if $\left|\sigma(a)\right|=k $ for any $a\in \A$. A $k$-uniform morphism will also be called $k$-morphism. If $k=1$, then $\sigma$ is simply called a \textit{coding}.

If $\mathcal {A}= \mathcal{B}$ we can iterate the application of $\sigma$. Hence, if $a\in \mathcal A$, $\sigma^0(a)=a$, $\sigma^i(a)=\sigma(\sigma^{i-1}(a))$, for every $i\geq 1$.
Let $\sigma: \mathcal{A}\rightarrow \mathcal{A}^*$ be a morphism. The set $\mathcal A^* \cup \mathcal A ^{\N}$ is endowed with its natural topology. 

Roughly, two words are close if they have a long common prefix. 
We can thus extend the action of a morphism by continuity to $\mathcal A^* \cup \mathcal A ^{\N}$. 
Then, a word $\a\in \mathcal{A}^{\N}$ is \emph{a fixed point} of a morphism $\sigma$ if $\sigma(\a)=\a$.\\

A morphism $\sigma$ is \emph{prolongable} on $a\in \mathcal{A}$ if $\sigma (a)=aX$, for some $X\in \mathcal{A}^{+}:=\mathcal{A}^*\backslash \{\varepsilon\}$ such that $\sigma^k(X)\neq \varepsilon$, for any $k\in \N$. If $\sigma$ is prolongable then the sequence $(\sigma^i(a))_{i\geq 0}$ converges to the infinite word $$\sigma^{\infty}(a)=\lim_{i\rightarrow \infty}\sigma^i(a)=aX\sigma(X)\sigma^2(X)\sigma^3(X)\cdots.$$

With this notation, we now can cite an important theorem of Cobham, which gives a characterization of $k$-automatic sequences in terms of $k$-uniform morphisms.
\begin{theorem}[Cobham, 1972] Let $k\geq 2$. Then a sequence $\a=(a_i)_{i\geq 0}$ is $k$-automatic is and only if it is the image, under a coding, of a fixed point of a $k$-uniform morphism.
 \end{theorem}

\end{section}

\begin{section}{Proof of Theorem \ref{MUIN11}}\label{3.maintheorem}


\medskip

Theorem \ref{MUIN11} is an easy consequence of the more precise result established in 
Theorem \ref{muin112}. In order to prove Theorem \ref{muin112}, we first establish our approximation lemma 
which is the analog of a classical result in Diophantine approximation (Lemma \ref{approximation_lemma}). Then 
we show how to construct, starting with an arbitrary algebraic Laurent series $f$ with coefficients in a finite field,  
an infinite sequence of rational approximations of $f$ satisfying the assumptions of our approximation lemma. 
We thus deduce the expected upper bound for the irrationality exponent of $f$. 

All along this section, we provide comments and remarks allowing one to improve in most cases 
this general upper bound (see in particular Remark \ref{3.sameprefix} and Section \ref{3.Equivalentcondition}).

\begin{subsection}{Maximal repetitions in automatic sequences}\label{Maximal repetitions in automatic sequences}

Before stating our approximation lemma, we first recall a useful result, which will allow us later 
to control 
repetitive patterns occurring as prefixes of automatic sequences. The proof of the following lemma can be found in \cite{adamczewski_cassaigne} (Lemma 5.1, page 1356). Before stating it, we recall that the kernel $K_k$ of a 
$k$-automatic sequence $\a=(a_i)_{i\geq 0}$ is defined as the set of all subsequences of the form 
$(a_{k^ni+l})_{i\geq 0}$, where $n\geq 0$ and $0\leq l<k^n$.
Furthermore, we recall that by a result of Eilenberg 
a sequence $\a$ is $k$-automatic if and only if $K_k(\a)$ is finite.

\begin{lemma}\label{3.ker} Let $\a$ be a non-ultimately periodic $k$-automatic sequence defined on a alphabet $\mathcal A$. Let $U\in \A^*, V\in \A^* \setminus \{\varepsilon\}$ and $\omega \in \Q$ be such that $UV^{\omega}$ is a prefix of the sequence $\a$. Let $s$ be the cardinality of the $k$-kernel of $\a$. Then we have the following inequality:
$$\frac{\vert UV^{\omega}\vert}{\vert UV\vert } < k^s.$$
\end{lemma}
\end{subsection}

\begin{subsection}{An approximation lemma}

We start with the following result which is, in fact, an analog version of Lemma 4.1 in \cite{adamczewski_rivoal} for Laurent series with coefficients in a finite field. We also recall the proof, since it is not very long and it may also be of independent interest.

\begin{lemma}\label{approximation_lemma} Let $f(T)$ be a Laurent series with coefficients in $\F_q$. Let $\delta, \rho$ and $\theta$ be real numbers such that $0< \delta \leq \rho$ and $\theta \geq 1$. Let us assume that there exists a sequence $(P_n/Q_n)_{n\geq 1}$ of rational fractions with coefficients in $\F_q$ and some positive constants $c_0, c_1$ and $c_2$ such that
\begin{itemize}
\item[(i)] $$|Q_n| < |Q_{n+1}| \leq c_0 |Q_{n}|^{\theta}$$
\item[(ii)] $$\frac{c_1}{|Q_n|^{1+\rho}} \leq |f-\frac{P_n}{Q_n}| \leq \frac{c_2}{|Q_{n}|^{1+\delta}}.$$
\end{itemize}
Then, the irrationality measure $\mu(f)$ satisfies
\begin{equation}\label{muin} 1+\delta \leq \mu(f) \leq \frac{\theta(1+\rho)}{\delta}.
 \end{equation}
Furthermore, if we assume that there is $N \in \N^*$ such that for any $n\geq N$, $(P_n, Q_n)=1$, 
then, we have
$$1+\delta \leq \mu(f) \leq \mathrm{max }(1+\rho, 1+\frac{\theta}{\delta}).$$
In this case, if $\rho=\delta$ and $\theta \leq \delta^2$, then $\mu(f)=1+\delta$.
 \end{lemma}

\begin{proof}

 The left-hand side inequality is clear. We thus turn our attention to the second inequality. Let $P/Q \in \F_q(T)$ such that $\vert Q\vert$ is large enough. Then there exists a unique integer $n=n(Q)\geq 2$ such that 
\begin{equation} \label{applemma}
 \vert Q_{n-1}\vert < (2c_2 \vert Q \vert)^{\frac{1}{\delta}} \leq \vert Q_n \vert.
\end{equation}
If $P/Q\neq P_n/Q_n$ then 
$$\left| \frac{P}{Q}- \frac{P_n}{Q_n} \right| \geq \frac{1}{\vert Q Q_n \vert}, $$
and using Eq. (\ref{applemma}) and (ii) we get that 
\begin{equation}\nonumber
\left| f- \frac{P_n}{Q_n}\right| \leq \frac{c_2}{|Q_n|^{1+\delta}}=\frac{c_2}{|Q_n||Q_n|^{\delta}}\leq \frac{1}{2|Q||Q_n|}.
 \end{equation}
By the triangle inequality, we have that
$$\left| f-\frac{P}{Q} \right| \geq \left|\frac{P}{Q}-\frac{P_n}{Q_n} \right| - \left|f-\frac{P_n}{Q_n}\right|. $$
Eq. (i) together with Eq. (\ref{applemma}) imply that $\vert Q_n \vert \leq c_0 \vert Q_{n-1}\vert^{\theta}< c_0(2c_2 \vert Q \vert) ^{\theta / \delta} $. Thus, 
$$ \left| f- \frac{P}{Q}\right| \geq \frac{1}{2|Q||Q_n|} \geq \frac{1}{2|Q|c_0(2c_2 \vert Q \vert) ^{\theta / \delta}} \geq 
\frac{c_3}{\vert Q \vert ^{\frac{\theta(1+\rho)}{\delta}} }$$
since $1+\theta / \rho \leq \theta +\theta / \rho$ (because $\theta\geq 1$) and $c_3:=1/2c_0(2c_2)^{\theta/\delta}$.

On the other hand, if $P/Q=P_n/Q_n$, then 
$$\left| f- \frac{P}{Q}\right| = \left| f- \frac{P_n}{Q_n}\right|\geq \frac{c_1}{\vert Q_n\vert^{1+\rho}}\geq \frac{c_1}{(c_0 (2c_2\vert Q\vert)^{\theta/\rho})^{1+\rho}} =\frac{c_4}{\vert Q \vert ^{\frac{\theta (1+\rho)}{\delta}}}, $$
where $c_4=c_1/c_0^{1+\rho} (2c_2)^{\frac{\theta(1+\rho)}{\delta}}$.
This ends the proof.

%
\medskip

The case where $(P_n, Q_n)=1$ is treated in a similar way and we refer the reader to Lemma 4.1 in \cite{adamczewski_rivoal} (page 10). 
The proof consists, as previously, of two cases, but, when $P_n/Q_n=P/Q$ and $P/Q$ is reduced, then $Q_n=Q$; this permits to obtain an improved upper bound.
\end{proof}

Note that the second part of Lemma \ref{approximation_lemma} is also known as a Voloch's Lemma and for more details, we refer the reader to the original paper of \cite{Voloch88}, but also to Thakur's monograph \cite{thakur}, page 314.

\end{subsection}

\begin{subsection} {Construction of rational approximations via Christol's theorem}\label{3.construction}

Let $$f(T):=\sum_{i\geq 0}{a_iT^{-i}} \in \F_q[[T^{-1}]]$$ be an irrational algebraic Laurent series over $\F_q(T)$.

 We recall that, by Christol's theorem, the sequence $\a:=(a_i)_{i\geq 0}$ is $p$-automatic. 
According to Cobham's theorem, there exist $m\geq 1$, $$\sigma: \A_m \mapsto \A_m^*$$ a $p$-morphism and $$\varphi: \A_m \mapsto \F_q$$ a coding such that $\a=\varphi(\sigma^{\infty}(a))$ where $a \in \mathcal A_m$.

In all that follows, we let $f_{\a}(T)=\sum_{i\geq 0}{a_iT^{-i}}$ denote the Laurent series associated with the infinite word $\a=(a_i)_{i\geq 0}$.

 We also give the following definition for a polynomial associated with a finite word.

\begin{definition}\label{T.P_U}
 Let $U=a_0 a_1\cdots a_{k-1}$ be a finite word over a finite alphabet. We associate with $U$ the polynomial $P_{U}(T):=\sum_{j=0}^{k-1}{a_{k-1-j}T^j}$. If $U=\varepsilon$, we set $P_{ U}(T)=0$.
\end{definition}
 For example, if we consider the word $U=1020031\in \A_5$ then $$P_U(T)=T^6+2T^4+3T+1$$ is a polynomial with coefficients in $\F_5$.


\medskip
Using this notation, we have the following two lemmas.

\begin{lemma}\label{f_UV} Let $U, V$ be two finite words such that $|U|=k \in \N$ and $|V|=\ell \in \N$ and let $\a:=UV^{\infty}$. Then, 
 $$f_{\a}(T)=\frac{P_{ U}(T)(T^{\ell}-1)+P_{ V}(T)}{T^{k-1}(T^{\ell}-1)}.$$
If $k=0$, we have
$$f_{\a}(T)=\frac{TP_{ V}(T)}{T^{\ell}-1}.$$
\end{lemma}

\begin{proof} Let $U:=a_0a_1\cdots a_{k-1}$ and $V:=b_0b_1\cdots b_{\ell-1}$. Writing the associated Laurent series with $\a:=UV^{\infty}$, we have
$$f_{\a}(T)=(a_0+a_1T^{-1}+\cdots+a_{k-1}T^{-(k-1)})+(b_0T^{-k}+\cdots+b_{\ell-1}T^{-(k+l-1)})+\cdots $$
and then, factorizing $T^{-(k-1)}, T^{-k}, T^{-(k+\ell)}, T^{-(k+2\ell)}, \ldots$ and using the definition of $P_U(T)$ and $P_V(T)$ (see def. \ref{T.P_U}), we obtain
\begin{align*}f_{\a}(T)=&\;T^{-(k-1)}P_U(T)+T^{-k}(b_0+b_1T^{-1}+\cdots+b_{\ell-1}T^{-(\ell-1)})+\\
&+T^{-(k+\ell)}(b_0+\cdots+b_{\ell-1}T^{-(\ell-1)})+\cdots\\
=&\;T^{-(k-1)} P_U(T)+T^{-(k+\ell-1)}P_V(T)(1+T^{-\ell}+T^{-2\ell}+T^{-3\ell}+\cdots)\\
=&\; \frac{P_{ U}(T)(T^{\ell}-1)+P_{ V}(T)}{T^{k-1}(T^{\ell}-1)}.
\end{align*}
When $k=0$, that is, when $U=\varepsilon$, we just have to replace $k=0$ in the identity given above. 
\end{proof}

\begin{lemma}\label{3.value}
Let $\a=(a_i)_{i\geq 0}$ and $\b=(b_i)_{i\geq 0}$ be two infinite sequences over a finite alphabet, 
satisfying $a_i=b_i$ for $0\leq i \leq L-1$, where $L\in \N^*$. Then, we have
$$\vert f_{\a}-f_{\b} \vert \leq \frac{1}{\vert T \vert ^L}.$$
the equality being obtained when $a_L\neq b_L$.
\end{lemma}

\begin{proof} This lemma immediately follows from the definition of an ultrametric norm.
\end{proof}


We now construct a sequence of rational fractions $(P_n/Q_n)_{n\geq 0}$ satisfying the assumptions of Lemma \ref{approximation_lemma}. The approach we use appears in \cite{adamczewski_cassaigne} and is essentially based on the repetitive patterns occurring in automatic sequences.

The sequence $\a$ being $p$-automatic, the $p$-kernel is finite. We let $e$ denote the number of states of the minimal automaton generating $\a$ (in \textit{direct reading}) and $s$ the cardinality of the $p$-kernel. 
 Let consider a prefix $P$ of $\sigma^{\infty}(a)$ of length $e+1$. Observe that $e$ is greater than or equal to the cardinality of the internal alphabet of $\a$, that is, $\A_m$. It follows, from the pigeonhole principle, 
 that there exists a letter $b \in \A_m$ occurring at least twice in $P$. This means that there exist two (possibly empty) words $U'$ and $V'$ and a letter $b$ (both defined over $\A_m$) such that 
 $$P:=U'bV'b.$$
 Now, if we set $U:=U'$, $V:=bV'$, $\vert U\vert:=k$, $\vert V \vert: =\ell$, $\omega:=1+1/{\ell}$, 
 then $UV^{\omega}$ is a prefix of $\sigma^{\infty}(a)$.\\
 


Let $n\in \N$, $U_n:=\varphi(\sigma^n(U))$ and $V_n=\varphi(\sigma^n(V))$. Since $$\a=\varphi(\sigma^{\infty}(a))$$ then, for any $n\in \N$, $U_nV_n^{\omega}$ is a prefix of $\a$. Notice also that $ \vert U_n\vert=\vert U\vert p^n$ and $\vert V_n\vert=\vert V \vert p^n$ and the sequence $(\vert V_n\vert)_{n\geq 1}$ is increasing.

For any $n\geq 1$, we set 
\begin{equation}\label{3.Q_n}
Q_n(T)=T^{kp^n-1}(T^{\ell p^n}-1).
\end{equation}

Let $\textbf{c}_n$ denote the infinite word $U_nV_n^{\infty}$. There exists $P_n(T)\in \F_q[T]$ such that 
$$f_{\textbf{c}_n}(T)=\frac{P_n(T)}{Q_n(T)}.$$
More precisely, by Lemma \ref{f_UV}, the polynomial $P_n(T)$ may be defined by the following formula \begin{equation}\label{3.P_n}P_n(T)=P_{U_n}(T)(T^{\ell p^n}-1)+P_{ {V_n}}(T). \end{equation}

Since $\a$ and $\textbf{c}_n$ have the common prefix $U_nV_n^{\omega}$ it follows by Lemma \ref{3.value} that
\begin{equation}\label{frarat}
\left| f_{\a}-\frac{P_n}{Q_n}\right |\leq \frac{c_2}{\vert Q_n \vert^{\frac{k+\omega \ell}{k+\ell}}}, 
\end{equation}
where $c_2=\vert T \vert ^{\frac{k+l}{k+\omega l}}$.

Furthermore, if $\a$ and $\textbf{c}_n$ have the common prefix $U_nV_n^{\omega}$ and the $(k+\omega \ell)p^n+1$-th letters are different, then, by Lemma \ref{3.value}, Inequality (\ref{frarat}) becomes an equality. On the other hand, we also have the following result, which is an easy consequence of Lemma \ref{3.ker}.

\begin{lemma} Let $s$ be the cardinality of the $p$-kernel of the sequence $\a:=(a_i)_{i\geq 0}$. We have
$$\left| f_{\a}-\frac{P_n}{Q_n}\right | \geq \frac{1}{\vert Q_n\vert^{p^s}}.$$
\end{lemma}

\begin{proof} Using Lemma \ref{3.ker} we obtain
$$\vert U_nV_n^{\omega}\vert < p^s \vert U_nV_n\vert.$$

This implies that $\a$ and $\textbf{c}_n$ cannot have the same first $p^s \vert U_nV_n\vert $ digits. Hence 
\begin{align*}
\left|f_{\a}-\frac{P_n}{Q_n}\right| & \geq \frac{1}{\vert T \vert ^{(|U_n|+ |V_n|)p^s}} \\
 &=\frac{1}{(\vert T \vert |Q_n|)^{p^s}} \geq \frac{c_1}{|Q_n|^{p^s}}.\end{align*}
 where $c_1:=1/\vert T \vert ^{p^s}$.
\end{proof}

This shows that $(P_n/Q_n)_{n\geq 1}$ satisfies the assumptions of Lemma \ref{approximation_lemma} with $\theta=p$, $\rho=p^s-1$ and $\delta=\frac{(\omega-1)\ell}{k+\ell}$. With this notation, we obtain the following theorem. 

\begin{theorem}\label{muin112}
Let $f_{\a}(T)=\sum_{i\geq 0}{a_iT^{-i}} \in \F_q[[T^{-1}]]$ be an algebraic Laurent series over $\F_q(T)$. Let $k, l, \omega, s$ be the parameters of $f_{\a}$ defined above. Then, the irrationality measure of $f_{\a}$, $\mu(f_{\a})$, satisfies the following inequality:
\begin{equation}\label{muin12}
 \frac{k+\omega\ell}{k+\ell}\leq \mu(f_{\a}) \leq \frac{p^{s+1}(k+\ell)}{(\omega-1)\ell}.\end{equation}
\end{theorem}

\begin{rem}\label{3.sameprefix} 
If, for every $n$, $\a$ and $\textbf{c}_n$ have the same first $(k+\omega \ell) p^n$ digits, while the $(k+\omega \ell) p^n+1$-th digits are different, then we have 
$$\left|f_{\a}-\frac{P_n}{Q_n}\right|=\frac{c_2}{\vert Q_n \vert^{1+\delta}}.$$
Thus, Inequality (\ref{muin12}) does not depend on $s$ (the cardinality of $p$-kernel) anymore. More precisely, in this case, $P_n$ and $Q_n$ satisfy Lemma \ref{approximation_lemma} with $\theta=p$, $\rho=\delta=\frac{(\omega-1)\ell}{k+\ell}$ and we have
$$\frac{k+\omega\ell}{k+\ell} \leq \mu(f_{\a}) \leq \frac{p(k+\omega\ell)}{(\omega-1)\ell}.$$
Moreover, if there exists $N$ such that, for any $n\geq N$, $(P_n, Q_n)=1$, then
$$\frac{k+\omega\ell}{k+\ell} \leq \mu(f_{\a}) \leq \mathrm{max}\left (\frac{k+\omega\ell}{k+\ell}, 1+\frac{p(k+\ell)}{(\omega-1)\ell}\right).$$
If $U=\varepsilon$, that is $k=0$, then
 $$\omega \leq \mu(f_{\a}) \leq p\frac{\omega}{\omega-1}.$$ Furthermore, if $\omega-1\geq \sqrt p$ and $(P_n, Q_n)=1$, then $\mu(f_{\a})=\omega$. All this explains why, in many cases, the general upper bound 
 we obtained in Theorem \ref{MUIN11} can be significatively improved. 
\end{rem}

\begin{proof}[Proof of Theorem \ref{MUIN11}] 
By construction, $\omega=1+1/{\ell}$ and $k+\ell \leq d$. By Theorem \ref{muin112}, it immediately follows that $\mu(f_{\a})\leq p^{s+1}e$.
\end{proof}
\end{subsection}

\begin{subsection}{An equivalent condition for coprimality of $P_n$ and $Q_n$}\label{3.Equivalentcondition}

We have seen in Remark \ref{3.sameprefix} that, in the case where the numerator $P_n$ 
and the denominator $Q_n$ of our rational approximations are relatively prime, the 
bound for the irrationality exponent obtained in Theorem \ref{MUIN11} can be significantly improved. 
This serves as a motivation for this section,  which is devoted to the coprimality of 
$P_n$ and $Q_n$. 

First, let us recall the following result, which is an easy consequence of the fact that the greatest common divisor of two polynomials, defined over a field $\K$, also belongs to $\K$.

\begin{lemma}
 Let $P, Q \in \F_q[T]$. Then $(P, Q)=1$ over $\F_q[T]$ if and only if $(P, Q)=1$ over $\overline{\F}_p[T]$. 
\end{lemma}
We recall that $\overline{\F}_p$ is the classical notation for an algebraic closure of $\F_p$.

Let $n\in \N^*$ and $Q_n(T)=T^{kp^n-1}(T^{\ell p^n}-1) \in \F_q[T]$. Since we work in characteristic $p$, we have that 
$$Q_n(T)=T^{kp^n-1}(T^{\ell}-1)^{p^n}.$$
Now, let $P$ be an arbitrary polynomial with coefficients in $\F_q$. Then $(P, Q_n)=1$ if and only if
$(P(T), T)=1$ and $(P(T), T^{\ell}-1)=1$. In other words, $(P, Q_n)=1$, if and only if $P(0)\neq 0$ and $P(a)\neq 0$ for all $a\in \overline{\F}_p$ such that $a^{\ell}=1$.


 Therefore, we easily obtain the following lemma, which will simplify the study of the coprimality of polynomials $P_n$ and $Q_n$, by using some properties of $P_{ {U_n}}$ and $P_{ {V_n}}$. (We recall that $U_n=\varphi(\sigma^n(U))$ and $V_n=\varphi(\sigma^n(V))$, where $U$ and $V$ are introduced in Section \ref{3.construction}.)
\begin{lemma}\label{3.(P_n, Q_n)} Let $n\in \N^*$ and $P_n$, $Q_n$ defined in Eq. (\ref{3.P_n}) and (\ref{3.Q_n}). Then, $(P_n, Q_n)=1$ over $\F_q(T)$ if and only if we have 
\begin{itemize}
\item [(i)] $P_{ {U_n}}(0)\neq P_{ {V_n}}(0)$, 
 \item [(ii)] for any $a\in \overline{\F}_p$, such that $a^{\ell}=1$, $P_{ {V_n}}(a)\neq 0$.
\end{itemize}
\end{lemma}

\end{subsection}
\end{section}

\begin{section}{Matrices associated with morphisms}\label{Matrix associated with morphisms}

The purpose of this section is to give an approach which will allow one to compute the polynomials $P_{ {U_n}}(T)$ and $P_{ {V_n}}(T)$, described in the previous section. In particular, we show that, if $\alpha \in \overline{\F}_p$, the sequences $(P_{ {U_n}}(\alpha))_{n\geq 1}$ and $(P_{ {V_n}}(\alpha))_{n\geq 1}$ are ultimately periodic. Lemma \ref{3.(P_n, Q_n)} implies that we have to test the coprimality of $P_n$ and $Q_n$ only for a finite number of index $n$.\\

Let $U= a_0a_{1}\cdots a_{k-1}$ be a finite word on $\A_m$ and let $i\in \A_m$. We let $\P_{ U}(i)$ denote the set of positions of $i$ in the word $\overline U$; or, simply, $\P_i$ if there is no doubt about $U$.
\begin{defi}\label{v_U} We associate with $U$ the row vector $v_U(T)=(\beta_{U, j}(T))_{0\leq j \leq m-1}$ with coefficients in $\F_p[T]$ where, for any $j \in\A_m$, $\beta_{U, j}$ is defined as:
\begin{equation}\label{3.definit}\beta_{U, j}(T)= \begin{cases} \sum_{l\in \P_{j}}T^l, \text{if $j$ occurs in $U$}, \\
0, \text{ otherwise}.
\end{cases}\end{equation}

\end{defi}

\begin{ex} We consider $U=1020310 \in \A_5^*$. Then $\P_0=\{0, 3, 5\}$, $\P_1=\{1, 6\}$, $\P_2=\{4\}$, $\P_3=\{2\}$ and $\P_4=\emptyset$. The vector associated with $U$ is $$v_U(T)=(1+T^3+T^5, T+T^6, T^4, T^2, 0).$$
We also recall that $P_U(T)=T^6+2T^4+3T^2+T$ (see Definition \ref{T.P_U}) and we observe that \begin{displaymath}P_U(T)=v_U(T)\left( \begin{array}{c} 0\\1\\2\\3\\ 4 \end{array} \right) .\end{displaymath}
\end{ex}

\begin{defi}\label{3.matrixmorphism} Let $\sigma: \A_m \mapsto \A_m^*$ be a morphism. We associate with $\sigma$ the $m \times m$ matrix $M_{\sigma}(T)$ with coefficients in $\F_p[T]$ defined by
\begin{displaymath}
 M_{\sigma}(T)= 
 (\beta_{\sigma(i), j}(T))_{ 0 \leq i, j \leq m-1}.
\end{displaymath}
\end{defi}

\begin{ex}
 Let $\sigma: \A_3 \mapsto \A_3^*$, $\sigma(0) = 010, \sigma(1) = 2101$ and $\sigma(2) =00211$.
Then 
\begin{displaymath}
 M_{\sigma}(T)= 
\left( \begin{array}{cccc}
T^2+1 & T& 0 \\
T & T^2+1 & T^3 \\
T^4+T^3 & T+1 & T^2 \\
\end{array} \right).
\end{displaymath}
\end{ex}

It is not difficult to see that such matrices satisfy some interesting general properties as claimed 
in the following remarks.


\begin{rem} The matrix $M_{\sigma}(1)$ is the reduction modulo $p$ of the so-called incidence matrix associated with the morphism $\sigma$. This matrix satisfy some very nice properties and has been the subject of considerable study (see for instance \cite{Queff-SDS}).
\end{rem}

\begin{rem} \label{composition} If $\sigma_1$ and $\sigma_2$ are two $p$-morphisms over $\A_m$ then we have that 
$$M_{\sigma_1 \circ \sigma_2}(T)=M_{\sigma_2}(T^p)M_{\sigma_1}(T).$$
\end{rem}

Now, our main goal is to prove that, if $\alpha \in \overline{\F}_p$, the sequences $(P_{ {U_n}}(\alpha))_{n\geq 1}$ and $(P_{ {V_n}}(\alpha))_{n\geq 1}$ are ultimately periodic. This will be the subject of Proposition \ref{3.ultperiod}. In order to prove this, we will need the following auxiliary results.

\begin{lemma}\label{P_U_n}
Let $\sigma : \A_m \mapsto \A_m^*$ be a $p$-morphism and $U=a_0\cdots a_{k-1} \in \A_m^*$. For any $n\in \N$ we denote $U_n=\sigma^n(U)=\sigma^n(a_{0})\cdots \sigma^n(a_{k-1})$.
We have $$P_{ {U_n}}(T)=v_U(T^{p^n}) R_n(T), $$ where, for any $n \in \N$, 
 \begin{displaymath}
R_n(T) =
\left( \begin{array}{c}
P_{\sigma^n(0)}(T) \\
P_{\sigma^n(1)}(T) \\
\vdots \\
P_{\sigma^n(m-1)}(T)
\end{array} \right).
\end{displaymath}

\end{lemma}

\begin{proof}
 Since $U_n=\sigma^n(a_0)\cdots \sigma^n(a_{k-1})$, we infer from the fact that $\sigma$ is a $p$-morphism, 
 that
$$ P_{\sigma^n(U)}(T)= P_{\sigma^n(a_{0})}(T) T^{(k-1) \cdot p^n}+P_{\sigma^n(a_{1})}(T) T^{(k-2) \cdot p^n}+ \cdots +P_{\sigma^n(a_{k-1})}(T).$$
Hence there exists a vector $S_n(T)=(s_0(T^{p^n}), s_1(T^{p^n}), \ldots, s_{m-1}(T^{p^n}))$, where $s_i(T)$, $0\leq i \leq m-1$, are some polynomials with coefficients $0$ or $1$, such that
$$P_{\sigma^n(U)}(T)= S_n(T) \left( \begin{array}{c}
P_{\sigma^n(0)}(T) \\
P_{\sigma^n(1)}(T) \\
\vdots \\
P_{\sigma^n(m-1)}(T)
\end{array} \right)$$
and $S_n(T)=S_0(T^{p^n})$. For $n=0$, the equality above becomes 
$$P_{ U}(T)=S_0(T)\left(\begin{array}{c}
0 \\
1\\
\vdots \\
m-1
\end{array} \right).$$
By Definitions \ref{v_U} and \ref{T.P_U}, we deduce that $S_0(T)=v_U(T)$. This ends the proof.
\end{proof}

\begin{ex} Let $\sigma$ be a $3$-morphism over $\A_2$ and $U=10100$. Then, for any $n\in \N$, $$\sigma^n(U)=\sigma^n(1) \sigma^n(0) \sigma^n(1) \sigma^n(0) \sigma^n(0) \text{ and } v_U(T)=(T^3+T+1, T^4+T^2).$$
Hence, 
\begin{align*}
P_{\sigma^n(U)}(T)=& P_{\sigma^n(1)}(T) T^{4 \cdot 3^n}+P_{\sigma^n(0)}(T) T^{3 \cdot 3^n}+ \\
&+P_{\sigma^n(1)}(T) T^{2 \cdot 3^n}+P_{\sigma^n(0)}(T) T^{ 3^n}+P_{\sigma^n(0)}(T) \\
= &(T^{3 \cdot 3^n} +T^{3^n}+1, T^{4 \cdot 3^n}+T^{2 \cdot 3^n}) \left( \begin{array}{c}
P_{\sigma^n(0)}(T) \\
P_{\sigma^n(1)}(T) \\
\end{array} \right)\\
=&v_U(T^{3^n})R_n(T).
\end{align*}
\end{ex}

\begin{lemma}
 Let $n\in \N$ and let $\sigma$ be a $p$-morphism over $\A_m$. We have
\begin{equation}\nonumber
 R_{n+1}(T)=M_{\sigma}(T^{p^n}) R_{n}(T), 
\end{equation}
where $M_{\sigma}(T)$ is the matrix associated with $\sigma$, as in Definition \ref{3.matrixmorphism}.
\end{lemma}

\begin{proof}
Let $\sigma$ be defined as follows
\begin{equation*}
 \left\{
 \begin{array}{ccc}
  \sigma(0)&=&a_{0}^{(0)}a_{1}^{(0)}\cdots a_{p-1}^{(0)}\\
  \sigma(1)&=&a_{0}^{(1)}a_{1}^{(1)}\cdots a_{p-1}^{(1)}\\
  \vdots & \vdots & \vdots \\
  \sigma(m-1)&=&a_{0}^{(m-1)}a_{1}^{(m-1)}\cdots a_{p-1}^{(m-1)}, 
 \end{array}
 \right.
\end{equation*}
where $a_i^{(j)} \in \A_m$, for any $i\in \{ 0, 1 , \ldots, p-1\}$ and $j\in \{ 0, 1, \ldots, m-1\} $. Then, for any $j\in \{ 0, 1, \ldots, m-1\}$ and $n\in \N^*$, we have
$$ \sigma^{n+1}(j)=\sigma^n(\sigma(j))=\sigma^n(a_{0}^{(j)}a_{1}^{(j)}\cdots a_{p-1}^{(j)})=
\sigma^n(a_{0}^{(j)})\sigma^n(a_{1}^{(j)})\cdots \sigma^n(a_{p-1}^{(j)}).$$
Hence
$$P_{\sigma^{n+1}(j)}(T)=P_{\sigma^{n}(a_{0}^{(j)})}(T)T^{(p-1)p^n}+\cdots+ P_{\sigma^n(a_{p-1}^j)}(T).$$
It follows by Lemma \ref{P_U_n} that
\begin{displaymath}
\left( \begin{array}{c}
P_{\sigma^{n+1}(0)}(T) \\
P_{\sigma^{n+1}(1)}(T) \\
\vdots \\
P_{\sigma^{n+1}(m-1)}(T)
\end{array} \right)= (\beta_{\sigma(i), j}(T^{p^n}))_{0\leq i, j \leq m-1}
\left( \begin{array}{c}
P_{\sigma^n(0)}(T) \\
P_{\sigma^n(1)}(T) \\
\vdots \\
P_{\sigma^n(m-1)}(T)
\end{array} \right), 
\end{displaymath}
that is, $R_{n+1}(T)=M_{\sigma}(T^{p^n})R_n(T)$, which ends the proof.

\end{proof}

\begin{rem} \label{3.matrix-morphism}In particular, if $n=0$ in the previous lemma, we obtain the 
following equalities 
\begin{displaymath}
\left( \begin{array}{c}
P_{\sigma(0)}(T) \\
P_{\sigma(1)}(T) \\
\vdots \\
P_{\sigma(m-1)}(T)
\end{array} \right)= (\beta_{\sigma(i), j}(T))_{0\leq i, j \leq m-1}
\left( \begin{array}{c}
0 \\
1 \\
\vdots \\
m-1
\end{array} \right)=M_{\sigma}(T)R_0(T), 
\end{displaymath}
and this being true for any $p$-morphism $\sigma$ defined over $\A_m$.
Therefore, we observe that given a matrix $M$ of this form, we can find only one morphism whose matrix is $M$.

Notice also that, if $\varphi$ is a coding defined over $\A_m$, we have a similar identity
\begin{displaymath}
\left( \begin{array}{c}
P_{\varphi(\sigma(0))}(T) \\
P_{\varphi(\sigma(1))}(T) \\
\vdots \\
P_{\varphi(\sigma(m-1))}(T)
\end{array} \right)= (\beta_{\varphi(\sigma(i)), j}(T))_{0\leq i, j \leq m-1}
\left( \begin{array}{c}
\varphi(0) \\
\varphi(1) \\
\vdots \\
\varphi(m-1)
\end{array} \right), 
\end{displaymath}

\end{rem}

The following corollaries immediately yield. 

\begin{coro}\label{3.R_n}
 Let $n\in \N^*$ and let $\sigma$ be a $p$-morphism defined on $\A_m$. We have
\begin{equation}\nonumber
 R_{n}(T)=M_{\sigma}(T^{p^{n-1}}) M_{\sigma}(T^{p^{n-2}}) \cdots M_{\sigma}(T) \left( \begin{array}{c}
0 \\
1 \\
\vdots \\
m-1
\end{array} \right), 
\end{equation}
where $M_{\sigma}(T)$ is the matrix associated with $\sigma$, as in Definition \ref{3.matrixmorphism}.
\end{coro}

\begin{coro}\label{3.Msigman}
Let $\sigma$ be a $p$-morphism defined on $\A_m$. Then for any $n\in \N^*$ we have
\begin{equation}\nonumber
 M_{\sigma^n}(T)=M_{\sigma}(T^{p^{n-1}}) M_{\sigma}(T^{p^{n-2}}) \cdots M_{\sigma}(T)
\end{equation}
where $M_{\sigma}(T)$ is the matrix associated with $\sigma$, as in Definition \ref{3.matrixmorphism}.
\end{coro}

\begin{coro}\label{3.a_F_p} Let $a\in \F_p$. Then for any $n\in \N$ 
we have $M_{\sigma^n}(a)=M_{\sigma}^n(a)$.
\end{coro}

%
%

\begin{prop} \label{M_kr}Let $p$ be a prime, $q$ a power of $p$ and $\sigma: \A_m \mapsto \A_m^*$ be a $p$-morphism. 
Let $\alpha \in \F_{r}$, where $r=p^t$, $t\in \N^*$. Then for any positive integer $k$ we have 

\begin{equation}\label{M_sigma_kr} M_{\sigma^{kt}}(\alpha)=(M_{\sigma^t}(\alpha))^k.
\end{equation}

\end{prop}

\begin{proof}We argue by induction on $k$. Obviously, this is true for $k=1$. We suppose that Eq. (\ref{M_sigma_kr}) is satisfied for $k$ and we prove it for $k+1$. Using Corollary \ref{3.Msigman} 
and the fact that $\alpha^{r}=\alpha$ we obtain that
\begin{align*}
M_{\sigma^{(k+1)t}}(\alpha)&=\underbrace{M_{\sigma}(\alpha^{p^{kt+t-1}})\cdots M_{\sigma}(\alpha^{p^t})} M_{\sigma}(\alpha^{p^{t-1}})\cdots M_{\sigma}(\alpha)\\
&= \underbrace{M_{\sigma}(\alpha^{p^{kt-1}})\cdots M_{\sigma}(\alpha)}\underbrace{ M_{\sigma}(\alpha^{p^{t-1}})\cdots M_{\sigma}(\alpha)}\\
&=M_{\sigma^{kt}}(\alpha)M_{\sigma^t}(\alpha) =(M_{\sigma^{t}}(\alpha))^{k+1}.
\end{align*}
\end{proof}

\begin{prop}\label{3.ultperiod} Let $p$ be a prime, $q$ a power of $p$ and $U=a_{k-1}\cdots a_0 \in \A_m^*$. Let $\sigma: \A_m \mapsto \A_m^*$ be a $p$-morphism and $\varphi : \A_m \mapsto \F_q$ a coding. Let $\alpha \in \F_{r}$, where $r=p^t $, $r\in \N^*$. Then the sequence $(P_{\varphi(\sigma^n(U))}(\alpha))_{n\geq 0}$ is ultimately periodic.
\end{prop}

\begin{proof} First, notice that, as in Lemma \ref{P_U_n}, we have 
$$P_{\varphi(\sigma^n(U))}(T)=v_U(T^{p^n})
 \left( \begin{array}{c}
P_{\varphi(\sigma^n(0))}(T) \\
P_{\varphi(\sigma^n(1))}(T) \\
\vdots \\
P_{\varphi(\sigma^n(m-1))}(T)
\end{array} \right).$$

By Remark \ref{3.matrix-morphism}, we have that
\begin{displaymath}
\left( \begin{array}{c}
P_{\varphi(\sigma^n(0))}(T) \\
P_{\varphi(\sigma^n(1))}(T) \\
\vdots \\
P_{\varphi(\sigma^n(m-1))}(T)
\end{array} \right)=
M_{\sigma^n}(T) \left( \begin{array}{c}
\varphi(0) \\
\varphi(1) \\
\vdots \\
\varphi(m-1)
\end{array} \right).
\end{displaymath}

Hence
$$P_{\varphi(\sigma^n(U)}(\alpha)=v_U(\alpha^{p^n})M_{\sigma^n}(\alpha) \left( \begin{array}{c}
\varphi(0) \\
\varphi(1) \\
\vdots \\
\varphi(m-1)
\end{array} \right).$$

Clearly, the sequence $(v_{U}(\alpha^{p^n}))_{n\geq 0}$ is periodic with a period less than or equal to $t$ since $v_U(\alpha^{p^{n+t}})=v_U(\alpha^{p^n})$, for any $n\in \N^*$.
We now prove that the sequence $(M_{\sigma^n}(\alpha))_{n\geq 0}$ is ultimately periodic. \\
Since $\alpha \in \F_{p^t}$, we have $\alpha^{p^t}=\alpha$ and thus, by Corollary \ref{3.Msigman}, for any $k$ and $n\in \N$, we obtain 
$$M_{\sigma^{n+kt}}(\alpha)=M_{\sigma^n}(\alpha)M_{\sigma^{kt}}(\alpha).$$
Therefore, by Proposition \ref{M_kr} we have, for any $k\in \N$, 
$$M_{\sigma^{n+kt}}(\alpha)=M_{\sigma^n}(\alpha)M_{\sigma^{kt}}(\alpha)=M_{\sigma^n}(\alpha) (M_{\sigma^t}(\alpha))^k.$$

Since $M_{\sigma^t}(\alpha)$ is a $m\times m$ matrix with coefficients in a finite field, there exist two positive integers $m_0$ and $n_0$, $m_0\neq n_0$ (suppose 
that $m_0 < n_0$) such that $M_{\sigma^t}(\alpha)^{m_0}=M_{\sigma^t}(\alpha)^{n_0}$. This implies that
$$M_{\sigma^{n+m_0t}}(\alpha)=M_{\sigma^n}(\alpha)(M_{\sigma^{t}}(\alpha))^{m_0}=M_{\sigma^n}(\alpha) (M_{\sigma^{t}}(\alpha))^{n_0}=M_{\sigma^{n+n_0t}}(\alpha)$$ and thus, the sequence $(M_{\sigma^n}(\alpha))_{n\geq 0}$ is ultimately periodic, with pre-period at most $m_0t$ and period at most $(n_0-m_0)t$. Since $(v_U(\alpha^{p^n}))_{n\geq 0}$ is periodic with period at most $t$, it follows that $(P_{\varphi(\sigma^n(U))}(\alpha))_{n\geq 0}$ is ultimately periodic (with pre-period at most $m_0t$ and period at most $(n_0-m_0)t$). 
This ends the proof.
\end{proof}
\end{section}

\begin{rem}
All the properties (we have proved here) of these matrix associated with morphisms are still true when replacing $p$-morphisms by  $p^s$-morphisms, for any $s\in \N^*$.
\end{rem}

\begin{section}{Examples}\label{3.examples}
 In Theorem \ref{MUIN11} we give a general upper bound for the irrationality exponent of 
 algebraic Laurent series with coefficients in a finite field. In many cases, 
 the sequence of rational approximations $(P_n/Q_n)_{n \geq 0}$ we construct turns out to 
 satisfy the conditions (i) and (ii) of Lemma \ref {3.(P_n, Q_n)}. This naturally gives rise to a much better estimate, as hinted in Remark \ref{3.sameprefix}. In this section, we illustrate this claim with few examples 
 of algebraic Laurent series for which the irrationality exponent is exactly computed or at least well estimated.


\begin{ex}

Let us consider the following equation over $\F_2(T)$
\begin{equation}\label{Exeq1}
 X^4+X+\frac{T}{T^4+1}=0.
\end{equation}
This equation is related to the Mahler algebraic Laurent series, previously mentioned.
Let $E_1=\{ \alpha \in \F_2((T^{-1})), \; \; \vert \alpha \vert <1 \}$.
We first notice that Eq. (\ref{Exeq1}) has a unique solution $f$ in $E_1$. This can be obtained by showing that the application 
$$
\begin{array}{lll}
h: &E_1 & \longmapsto E_1 \\ 
 &X & \displaystyle\longmapsto X+\frac{T}{T^4+1}
\end{array}$$
is well defined and is a contracting map from $E_1$ to $E_1$.
Then the fixed point theorem in a complete metric space implies that the equation $h(X)=X$, which is equivalent to Eq. (\ref{Exeq1}), has a unique solution in $E_1$. Let $f(T):=\sum_{i\geq 0}{a_i T^{-i}}$  denote this solution  (with $a_i=0$, since $f$ belongs to $E_1$).

The second step is to find the morphisms that generate the sequence of coefficients of $f$, as in Cobham's theorem. Notice that there is a general method that allows one to obtain these morphisms when we know the algebraic equation. In this example, we try to describe the important steps of this method and we will give further details later.

By replacing $f$ in Eq. (\ref{Exeq1}) and using the fact that $f^4(T)=\sum_{i\geq 0}{a_i T^{-4i}}$ we easily obtain the following relations between the coefficients of $f$
\begin{align}
& a_{i+1}+a_i+a_{4i+4}+a_{4i}=0, \label{ex1al1}\\
&a_1=0, a_2=0, a_3=1, \label{ex1al2} \\
& a_{i+4}+a_i=0, \text{ if } i \not \equiv 0 [4]. \label{ex1al3}
\end{align}
By Eq. (\ref{ex1al2}) and (\ref{ex1al3}), we get that $a_{4i+1}=0$, $a_{4i+2}=0$ and $a_{4i+3}=1$, for any $i\geq 0$.
From Eq. (\ref{ex1al1}) we deduce that 
\begin{align*}
& a_{16i+4}=a_{16i+8}=a_{16 i}+a_{4i}, \\
 & a_{16i+12}=a_{16i+8}+1=a_{16 i}+a_{4i}+1.\\
\end{align*}
This implies that the $4$-kernel of $\a:=(a_i)_{i\geq 0}$ is the following set
$$K_4(\a)=\{(a_i)_{i\geq 0}, (a_{4i})_{i\geq 0}, (a_{16i})_{i\geq 0}, (0), (1)\}.$$
Consequently, the $4$-automaton generating $\a$ is 

\begin{figure}[h]
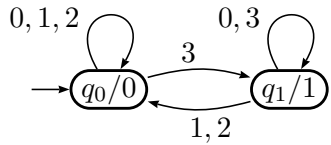

\centering

\VCDraw{%
  \begin{VCPicture}{(0, -1.5)(4, 2)}
\StateVar[q_0/0]{(0, 0)}{a} \StateVar[q_1/1]{(4, 0)}{b}
\Initial[w]{a}
\LoopN{a}{0, 1, 2}
\LoopN{b}{0, 3}
\ArcL{a}{b}{3}
\ArcL{b}{a}{1, 2}
\end{VCPicture}
}
\caption{A $4$-automaton recognizing $\a$}
\label{ex1fig}
\end{figure}

Once we have the automaton, there is a general approach to obtain the morphisms that generate an automatic sequence. More precisely, the proof of Cobham's theorem precisely describes this process. The reader may consult the original article of Cobham \cite{Cobham} or the monograph \cite{Allouche_Shallit} (Theorem 6.3.2, page 175). Following this approach, we obtain that $\a=\sigma^{\infty}(0)$, where $\sigma$ is defined as follows
\begin{align*}
 \sigma(0)&=0001\\
 \sigma(1)&=1001.
\end{align*}

It is now possible to apply our approach described in the first part of the paper. We will prove the following result.
\begin{prop}
One has 
\begin{equation}\nonumber
 \mu(f)=3.
\end{equation}
\end{prop}

 Notice that, Mahler's theorem implies only that $\mu(f) \leq 4$ and Osgood's Theorem or Lasjaunias and de Mathan's Theorem cannot be applied since this Laurent series is clearly hyperquadratic.

\begin{proof}
We are going to introduce an infinite sequence of rational fractions $(P_n/Q_n)_{n\geq 1}$ converging to $f$. Since $\a$ begins with $0001$, we denote $V=0$ and for any $n\geq 1$, $V_n=\sigma^n(V)$. Hence $\a$ begins with $V_nV_nV_n$ for any nonnegative integer $n$. 

Since $\vert V_n \vert=4^n$, we set $Q_n(T)=T^{4^n}-1$. In Section \ref{3.construction}, we showed that there exists a polynomial $P_n(T) \in \F_2[T]$ such that:
$$\frac{P_n(T)}{Q_n(T)}=f_{V_n^{\infty}}(T).$$
The Laurent series expansion of $P_n/Q_n$ begins with
$$\sigma^n(0001)\sigma^n(0001)\sigma^n(0001)\sigma^n(0), $$
and we deduce that it begins with
$$\sigma^n(0001)\sigma^n(0001)\sigma^n(0001)0$$
while the sequence $\a$ begins with
$$\sigma^n(0001)\sigma^n(0001)\sigma^n(0001) 1.$$
Hence, the first $3\cdot 4^n$th digits of the Laurent series expansion of $P_n/Q_n$ and $f$ are the same, while the following coefficients are different.
According to Remark \ref{3.sameprefix}, we deduce that:
$$ 3 \leq \mu(f) \leq 6, $$
or, if for any $n\geq 1$ we have $(P_n, Q_n)=1$ then $$\mu(f)=3.$$

It thus remains to prove that $(P_n, Q_n)=1$ for every positive integer $n$. 
Let $n\geq 1$.  By Lemma \ref{3.(P_n, Q_n)} we have to prove that $P_n(1)\neq 0$, \textit{i.e.}, $P_{\sigma^n(0)}(1)\neq 0$. Remark \ref{3.matrix-morphism} implies that

\begin{displaymath}
\left( \begin{array}{c}
P_{\sigma^n(0)}(T) \\
P_{\sigma^n(1)}(T) 

\end{array} \right)=
M_{\sigma^n}(T) \left( \begin{array}{c}
0 \\
1 
\end{array} \right).
\end{displaymath}
Hence, $P_{ {\sigma^n(0)}}(1)\neq 0$ if and only if
\begin{displaymath}\left(\begin{array}{cc}
1 &
0
\end{array} \right)M_{\sigma^n}(1) \left(\begin{array}{c}
0 \\
1
\end{array} \right)\neq 0, 
\end{displaymath}
that is, if and only if
\begin{displaymath}\left(\begin{array}{cc}
1 &
0
\end{array} \right)M_{\sigma}^n(1) \left(\begin{array}{c}
0 \\
1
\end{array} \right)\neq 0.
\end{displaymath}
Indeed, by Corollary \ref{3.a_F_p}, we have $M_{\sigma^n}(1)=M_{\sigma}^n(1)$.
Since 
\begin{displaymath}
 M_{\sigma}(1)= 
\left( \begin{array}{cc}
1 & 1 \\
0 & 0 
\end{array} \right), 
\end{displaymath}
we deduce that $M_{\sigma^n}(1)=M_{\sigma}(1)$ and then, 
\begin{displaymath}\left(\begin{array}{cc}
1 &
0
\end{array} \right)\left( \begin{array}{cc}
1 & 1 \\
0 & 0 
\end{array} \right) \left(\begin{array}{c}
0 \\
1
\end{array} \right)\neq 0.
\end{displaymath}
Consequently, we obtain that $(P_n, Q_n)=1$. This ends the proof. 
\end{proof}

\end{ex}

\begin{rem}
Let $f(T)=\sum_{i\geq 0}{a_iT^{-i}}$ be a Laurent series with coefficients in a finite field $\F_q$, where $q$ is a power of $p$, and suppose that there is $P(T)\in \F_q[T]$ such that $P(f)=0$.
As mentioned before, there is a general approach that allows one to find the automaton that generates the infinite sequence $\a:=(a_i)_{i\geq 0}$. This consists of the following steps. First, there exists a polynomial $Q$ with coefficients in $\F_q$, of the form $Q(X)=\sum_{i\geq 0}{B_i(T)X^{p^i}}$, $B_0(T)\neq 0$ such that $Q(f)=0$. This is known as an Ore's polynomial and the existence is due to Ore's theorem (for a proof see, for example, \cite{Allouche_Shallit}, Lemma 12.2.3, page 355). Hence, the first step is to find such a Ore's polynomial vanishing $f$ (this is possible by passing $P$ to the power of $p$ as many times as we need). The second step is to find some recurrent relations between the terms $a_i$ in order to find the kernel of $\a$. This is possible thanks to the Frobenius morphism. The third step is the construction of the automaton generating $\a$. Notice that a sequence is $k$-automatic if and only if its $k$-kernel is finite. The proof of this well-known result is explicit and we refer the reader to \cite{Allouche_Shallit} (Theorem 6.6.2, page 18). The last step is to find the morphisms generating $\a$, as described in Cobham's theorem. In order to do this, the reader may refer to the proof of Cobham's theorem, which is explicit as well.

\end{rem}

The following examples present different computations of irrationality exponents of Laurent power series over finite fields. We do not give the algebraic equations of them because the computation is quite long, but, as in the previous example (where we find the morphisms if we know the equation), there is a general approach that allows one to compute the equation of a Laurent series when we know the automaton generating the sequence of its coefficients. Indeed, by knowing the morphisms we can find the automaton (see the proof of Cobham's theorem), then knowing the automaton allows to find the kernel (see \cite{Allouche_Shallit}, Theorem 6.6.2, page 185) and also the relations between the coefficients. These relations allow one to find a polynomial that vanishes the Laurent series (the reader may consult the proof of Christol's theorem in \cite{christol_2} or \cite{Allouche_Shallit}-page 356, but also \cite{kedlaya}-where a generalisation of Christol's theorem is given). More precisely the polynomial that we compute from these relations is also an Ore's polynomial. Finally, we have to factor this polynomial and to check which is the irreducible factor vanishing our algebraic Laurent series.

\begin{ex}

We now consider the Laurent series $$f_{\a}(T)=\sum_{i\geq 0}{a_iT^{-i}} \in \F_2[[T^{-1}]],$$ where the sequence $\a:=(a_i)_{i\geq 0}$ is the image under the coding $\varphi$ of the fixed point of the $8$-uniform morphism $\sigma$,  $\varphi$ and $\sigma$ being defined as follows
\begin{displaymath} \begin{array}{ccc}
 \varphi(0)&=1\\
 \varphi(1)&=0\\
 \varphi(2)&=1 \end{array}
\quad \text{ and } \quad
\begin{array}{ccc}
 \sigma(0)&=00000122\\
 \sigma(1)&=10120011\\
 \sigma(2)&=12120021.
\end{array}
\end{displaymath}
Thus $\a=11111011111\cdots $.

\begin{prop}
One has 
$ \mu(f_{\a})=5$. \end{prop}

\begin{proof}
We are going to introduce an infinite sequence of rational fractions $(P_n/Q_n)_{n\geq 0}$ converging to $f_{\a}$. Since $\a$ begins with $00000$, we denote $V=0$ and, for any $n\geq 1$, $V_n=\sigma^n(0)$. Hence, $\a$ begins with $V_n^5$ for any $n\geq 1$. 
Now, we set $Q_n(T)=T^{8^n}-1$. In Section \ref{3.construction} we showed that there exists a polynomial $P_n(T) \in \F_2[T]$ such that:
$$\frac{P_n(T)}{Q_n(T)}=f_{V_n^{\infty}}(T).$$
 The Laurent series expansion of $P_n/Q_n$ begins with
$$\sigma^n(0) \sigma^n(0) \sigma^n(0) \sigma^n(0) \sigma^n(0) \sigma^n(0)$$
and we deduce that it begins with
$$\sigma^n(0) \sigma^n(0) \sigma^n(0) \sigma^n(0) \sigma^n(0) 0, $$
while the sequence $\a$ begins with
$$\sigma^n(0) \sigma^n(0) \sigma^n(0) \sigma^n(0) \sigma^n(0) 1.$$
Hence, the first $5 \cdot 8^n$th digits of the Laurent series expansion of $P_n/Q_n$ and $f_{\a}$ are the same, while the following coefficients are different.

Using Remark \ref{3.sameprefix}, we deduce that
$$ 5 \leq \mu(f_{\a}) \leq 10.$$
Moreover, if for any $n\geq 1$ we have $(P_n, Q_n)=1$, then we obtain $$ \mu(f_{\a})=5.$$

 We now prove that $(P_n, Q_n)=1$ for every $n\geq 1$. By Lemma \ref{3.(P_n, Q_n)} we have to prove that $P_n(1)\neq 0$, \textit{i.e.}, $P_{\varphi(\sigma^n(0))}(1)\neq 0$. By Remark \ref{3.matrix-morphism} we have that

\begin{displaymath}
\left( \begin{array}{c}
P_{\varphi(\sigma^n(0))}(T) \\
P_{\varphi(\sigma^n(1))}(T) \\
P_{\varphi(\sigma^n(2))}(T)
\end{array} \right)=
M_{\sigma^n}(T) \left( \begin{array}{c}
\varphi(0) \\
\varphi(1) \\
\varphi(2)
\end{array} \right)
\end{displaymath}
and by Corollary \ref{3.a_F_p} 
$$M_{\sigma^n}(1)=M_{\sigma}^n(1).$$
The matrix associated with $\sigma$, when $T=1$, is equal to the incidence matrix of $\sigma$: 
\begin{displaymath}
 M_{\sigma}(1)= 
\left( \begin{array}{ccc}
1 & 1 & 0 \\
1 & 0 & 1\\
0 & 1 & 1
\end{array} \right).
\end{displaymath}
Hence, by an easy computation, we obtain that, for any $j\geq 1$, 
\begin{displaymath}
 M_{\sigma}^{2j+1}(1)=\left( \begin{array}{ccc}
1 & 1 & 0 \\
1 & 0 & 1\\
0 & 1 & 1
\end{array} \right) \text{ and }
M_{\sigma}^{2j}(1)=\left( \begin{array}{ccc}
0 & 1 & 1 \\
1 & 0 & 1\\
1 & 1 & 0
\end{array} \right)
\end{displaymath}
and thus, for every $n\geq 1$, we have 
\begin{displaymath}\left(\begin{array}{ccccc}
1 &
0 &
0 
\end{array} \right)M_{\sigma}^n(1) \left(\begin{array}{c}
\varphi(0) \\
\varphi(1)\\
\varphi(2)
\end{array} \right)=1 \neq 0.
\end{displaymath}
Consequently, for every $n\geq 1$, $P_{\varphi(\sigma^n(0))}(1)\neq 0$ and thus $(P_n, Q_n)=1$.
\end{proof}

\end{ex}

\begin{ex}

We now consider the Laurent series $$f_{\a}(T)=\sum_{i\geq 0}{a_iT^{-i}} \in \F_3[[T^{-1}]],$$ where the sequence $\a:=(a_i)_{i\geq 0} $ is the fixed point beginning with zero of the following $3$-uniform morphism:
 \begin{align*}
 \sigma(0)&=010\\
 \sigma(1)&=102\\
 \sigma(2)&=122.
\end{align*}
Thus $\a=010102010\cdots$.

\begin{prop}
The irrationality exponent of the Laurent series $f_{\a}$ satisfies 
\begin{equation}\nonumber
 2.66\leq \mu(f_{\a}) \leq 2.81.
\end{equation}
\end{prop}

In this case, we are not able to compute the exact value of the irrationality exponent 
but the lower bound we found shows that the degree of $f_{\a}$ is greater than or equal to $3$. Hence our upper 
bound obviously improves on the one that could be deduced from the Liouville--Mahler's theorem.

\begin{proof}
 We are going to introduce an infinite sequence of rational fractions $(P_n/Q_n)_{n\geq 0}$ converging to $f_{\a}$. Since $\a$ begins with $010102$, we denote $V:=010102$ and for any $n\geq1$, $V_n:=\sigma^n(V)$. Hence $\a$ begins with $$\sigma^n(010102)\sigma^n(010102) \sigma^n(0101)$$ for any $n\geq 1$.
Now, we set $Q_n(T)=T^{2\cdot 3^n}-1$. There exists a polynomial $P_n(T) \in \F_3[T]$ such that:
$$\frac{P_n(T)}{Q_n(T)}=f_{V_n^{\infty}}(T).$$
The Laurent series expansion of $P_n/Q_n$ begins with
$$\sigma^n(010102)\sigma^n(010102)\sigma^n(0101) \sigma^n(0)$$
and we deduce that it begins with
$$\sigma^n(010102)\sigma^n(010102)\sigma^n(0101) 0, $$
while the sequence $\a$ begins with
$$\sigma^n(010102)\sigma^n(010102)\sigma^n(0101) 1.$$
Hence, the first $16\cdot 3^n$th digits of the Laurent series expansion of $P_n/Q_n$ and $f_{\a}$ are the same, while the following coefficients are different.
By Remark \ref{3.sameprefix}, we deduce that
$$ 2.66 \leq \mu(f_{\a}) \leq 4.81.$$
Furthermore, if for every $n\geq 1$ we have $(P_n, Q_n)=1$, then $$2.66\leq \mu(f_{\a}) \leq 2.81.$$

Let $n\geq 1$. We now prove that $(P_n, Q_n)=1$. By Lemma \ref{3.(P_n, Q_n)}, since $$Q_n(T)=(T-1)^{3^n}(T+1)^{3^n}, $$ we have to prove that, for all $n\geq 1$, $P_n(1)\neq 0$ and $P_n(-1)\neq 0$.

By definition of $(P_n(T))_{n\geq 0}$ (see Eq. (\ref{3.P_n})) we have 
$$ P_n(1)=P_{ {V_n}}(1) \text{ and } P_n(-1)=P_{ {V_n}}(-1).$$
Since $V_n=\sigma^n(010102)=\sigma^{n+1}(01)$ then, 
$$P_{ {V_n}}(T)=P_{\sigma^{n+1}(0)}(T) T^{3^{n+1}}+P_{\sigma^{n+1}(1)}(T).$$
Hence, 
\begin{align*}P_{ {V_n}}(1)&=P_{\sigma^{n+1}(0)}(1)+P_{\sigma^{n+1}(1)}(1) \\
P_{ {V_n}}(-1)&=-P_{\sigma^{n+1}(0)}(-1)+P_{\sigma^{n+1}(1)}(-1).
\end{align*}
Remark \ref{3.matrix-morphism} implies that
\begin{displaymath}
\left( \begin{array}{c}
P_{\sigma^n(0)}(T) \\
P_{\sigma^n(1)}(T) \\
 P_{\sigma^n(2)}(T)
\end{array} \right)=
M_{\sigma^n}(T) \left( \begin{array}{c}
0 \\
1 \\
2
\end{array} \right).
\end{displaymath}
If we now replace $T=1$ (respectively $T=-1$), we obtain that $P_n(1)\neq 0$ (respectively $P_n(-1)\neq 0$) if and only if
\begin{displaymath}\left(\begin{array}{ccc}
1 &
1 &
0
\end{array} \right)M_{\sigma}^n(1) \left(\begin{array}{c}
0 \\
1\\
2
\end{array} \right)\neq 0
\end{displaymath}
(respectively, 
\begin{displaymath}\left(\begin{array}{ccc}
-1 &
1 &
0
\end{array} \right)M_{\sigma}^n(-1) \left(\begin{array}{c}
0 \\
1\\
2
\end{array} \right)\neq 0).
\end{displaymath}
The matrix associated with $\sigma$ is 
\begin{displaymath}
 M_{\sigma}(T)= 
\left( \begin{array}{ccc}
T^2+1 & T & 0 \\
T & T^2 & 1 \\
0 & T^2 & T+1 
\end{array} \right).
\end{displaymath}
Hence, 
\begin{displaymath}
 M_{\sigma}(1)= 
\left( \begin{array}{ccc}
2 & 1 & 0 \\
1 & 1 & 1 \\
0 & 1 & 2 
\end{array} \right) \text{ and }
M_{\sigma}(-1)= 
\left( \begin{array}{ccc}
2 & -1 & 0 \\
-1 & 1 & 1 \\
0 & 1 & 0 
\end{array} \right).
\end{displaymath}

Notice that $ M_{\sigma}^2(\pm 1)=M_{\sigma}^4(\pm 1)$ and, by an easy computation, one can obtain that
the sequence $(P_{ {V_n}}(1))_{n\geq 0}$ is $2$-periodic and $(P_{ {V_n}}(-1))_{n\geq 0}$ is $1$-periodic; more precisely, $(P_{ {V_n}}(1))_{n\geq 0}=(12)^{\infty}$ and $(P_{ {V_n}}(-1))_{n\geq 0}=(1)^{\infty}$. This proves 
that $P_{ {V_n}}(1)$ and $P_{ {V_n}}(-1)$ never vanish, which ends the proof.

\end{proof}
\end{ex}

\begin{ex}

We now consider the Laurent series $$f_{\a}(T)=\sum_{i\geq 0}{a_iT^{-i}} \in \F_5[[T^{-1}]],$$ where the sequence $\a:=(a_i)_{i\geq 0}$ is the fixed point beginning with zero of the following $5$-uniform morphism:
 \begin{align*}
 \sigma(0)&=00043\\
 \sigma(1)&=13042\\
 \sigma(2)&=14201\\
\sigma(3)&=32411\\
\sigma(4)&=00144.
\end{align*}
Thus $\a=0004300043\cdots$.

\begin{prop}
One has 
\begin{equation}\nonumber
 \mu(f_{\a})=3.4.
\end{equation}
\end{prop}

Notice that, after Mahler's theorem, the degree of algebricity of $f_{\a}$ is greater than or equal than $4$.

\begin{proof}
 We are going to introduce an infinite sequence of rational fractions $(P_n/Q_n)_{n\geq 0}$ converging to $f_{\a}$. We can remark that the sequence $\a$ begins with $$00043000430004300144,$$ and thus, more generally, $\a$ begins with $$\sigma^n(00043)\sigma^n(00043)\sigma^n(00043) \sigma^n(00)\sigma^n(1)$$ for every $n\geq 1$. 
Now, we set $Q_n(T)=T^{5^n}-1$. There exists a polynomial $P_n(T) \in \F_5[T]$ such that:
$$\frac{P_n(T)}{Q_n(T)}=f_{\sigma^n(00043)^{\infty}}(T).$$
 The Laurent series expansion of $P_n/Q_n$ begins with
\begin{align*}
 &\sigma^n(00043) \sigma^n(00043) \sigma^n(00043) \sigma^n(00043)= \\
&= \sigma^n(00043) \sigma^n(00043) \sigma^n(00043) \sigma^n(00)\sigma^n(0) \sigma^n(43) 
\end{align*}
and we deduce that it begins with
$$\sigma^n(00043) \sigma^n(00043) \sigma^n(00043) \sigma^n(00) 0, $$
while the sequence $\a$ begins with
$$\sigma^n(00043) \sigma^n(00043) \sigma^n(00043) \sigma^n(00) 1.$$
Hence, the first $17 \cdot 5^n$th digits of the Laurent series expansion of $P_n/Q_n$ and $f_{\a}$ are the same, while the following coefficients are different. 

According to Remark \ref{3.sameprefix}, we deduce that
$$ 3.4 \leq \mu(f_{\a}) \leq 7.08.$$
Moreover, if for every $n\geq 1$ we have $(P_n, Q_n)=1$, then we obtain $$ \mu(f_{\a})=3.4 .$$

Let $n\geq 1$. We are going to prove that $(P_n, Q_n)=1$, for any $n\geq 1$. By Lemma \ref{3.(P_n, Q_n)}, it is necessary and sufficient to prove that $P_n(1)\neq 0$, \textit{i.e.}, $P_{\sigma^n(0)}(1)\neq 0$, which is equivalent to:

\begin{displaymath}\left(\begin{array}{ccccc}
1 &
0 &
0 &
0 &
0 
\end{array} \right)M_{\sigma}^n(1) \left(\begin{array}{c}
0 \\
1\\
2\\
3\\
4
\end{array} \right)\neq 0.
\end{displaymath}

Since the matrix associated with $\sigma$ is 
\begin{displaymath}
 M_{\sigma}(T)= 
\left( \begin{array}{ccccc}
T^4+T^3+T^2 & 0 & 0 & 1 & T \\
T^2 & T^4 & 1 &T^3 & T \\
T & T^4+1 & T^2 & 0 & T^3 \\ 
0 & T+1 & T^3 & T^4 & T^2 \\ 
T^4+T^3 & T^42 & 0 & 0 & T+1 
\end{array} \right).
\end{displaymath}
we obtain that
\begin{displaymath}
 M_{\sigma}(1)= 
\left( \begin{array}{ccccc}
3 & 0 & 0 & 1 & 1 \\
1 & 1 & 1 & 1 & 1\\
1 & 2 & 1 & 0 & 1 \\
0 & 2 & 1 & 1 & 1 \\
2 & 1 & 0 & 0 & 2. 
\end{array} \right).
\end{displaymath}

It follows by a short computation (for instance, using Maple) that the sequence $(M_{\sigma}^n(1))_{n\geq 1}$ is $20$-periodic and $(P_{ {V_n}}(1))_{n\geq 0}$ is $4$-periodic. Moreover $(P_{ {V_n}}(1))_{n\geq 0}=(2134)^{\infty}$. Hence $P_n(1) \neq 0$, for any $n\geq 1$.

\end{proof}

\end{ex}

\end{section}

\end{document}